\documentclass[11pt]{amsart}
\newtheorem{theorem}{Theorem}

\newtheorem{corollary}[theorem]{Corollary}

\newtheorem{lemma}[theorem]{Lemma}

\newtheorem{problem}[theorem]{Problem}
\newtheorem{proposition}[theorem]{Proposition}
\newtheorem{remark}[theorem]{Remark}

\usepackage[active]{srcltx} %SRC Specials for DVI search
\def\E{{\widehat{E}}}

\def\J#1#2#3{ \left\{ #1,#2,#3 \right\} }
\def\RR{{\mathbb{R}}}
\def\NN{{\mathbb{N}}}
\def\11{\textbf{$1$}}
\def\CC{{\mathbb{C}}}

\usepackage{enumerate}
\usepackage{amsmath}
\usepackage{amssymb}

\begin{document}

\title[Automatic continuity of derivations]{Automatic continuity of derivations on
C$^*$-algebras and JB$^*$-triples}
%Derivations from a JB$^*$-triple to a Jordan triple module}

\author[Peralta]{Antonio M. Peralta}
\email{aperalta@ugr.es}
\address{Departamento de An{\'a}lisis Matem{\'a}tico, Facultad de
Ciencias, Universidad de Granada, 18071 Granada, Spain.}

\author[Russo]{Bernard Russo}
\email{brusso@uci.edu}
\address{Department of Mathematics, UC Irvine, Irvine CA, USA}

\thanks{First author partially supported by
D.G.I. project no. MTM2008-02186, and Junta de Andaluc\'{\i}a
grants FQM0199 and FQM3737.}

\date{}

\begin{abstract}
We introduce the notion of a Jordan triple module and determine the
precise conditions under which every derivation from a
JB$^*$-triple $E$ into a Banach (Jordan) triple $E$-module is
continuous. In particular, every derivation from a real or complex
JB$^*$-triple into its dual space is automatically continuous.
Among the consequences, we prove that every triple derivation from
a C$^*$-algebra $A$ to a Banach triple $A$-module is continuous.
In particular, every Jordan derivation from $A$ to a Banach
$A$-bimodule is a derivation, a result which complements
a classical theorem due to B.E. Johnson and solves a problem
which has remained open for over ten years.
\end{abstract}

\maketitle
 \thispagestyle{empty}

\section{Introduction}

Results on automatic continuity of linear operators defined on
Banach algebras comprise a fruitful area of research intensively
developed during the last sixty years. The monographs
\cite{Sin76}, \cite{Dales78} and \cite{Dales00} review most of the
main achievements obtained during the last fifty years. In the words
of A.M. Sinclair (see \cite[Introduction]{Sin76}), ``the
continuity of a multiplicative linear functional on a unital
Banach algebra is the seed from which these results on the
automatic continuity of homomorphisms grew''.\smallskip

A linear mapping $D$ from a Banach algebra $A$ to a Banach
$A$-bimodule is said to be a \emph{derivation} if $D( a b) = D(a)
b + a D(b)$, for every $a,b$ in $A$. The pioneering  work of W. G.
Bade and P. C. Curtis (see \cite{BaCur}) studies the automatic
continuity of a module homomorphism between bi-modules over
$C(K)$-spaces. Some techniques developed in the just quoted paper
were exploited by J.R. Ringrose to prove that every (associative)
derivation from a C$^*$-algebra $A$ to a Banach $A$-bimodule $M$
is continuous (compare \cite{Ringrose72}). The case in which $M=A$
was previously treated by S. Sakai \cite{Sak60} by way of spectral
theory in $A$ ($=M$).\smallskip

A {\it Jordan derivation} from a Banach algebra $A$ into a Banach
$A$-module is a linear map $D$ satisfying $D(a^2) = a D(a) + D(a)
a,$ ($a\in A$), or equivalently, $D(ab+ba)=aD(b)+D(b)a + D(a)b
+bD(a),$ ($a,b\in A$). Sinclair proved that a bounded Jordan
derivation from a semisimple Banach algebra to itself is a
derivation, although this result fails for derivations of
semisimple Banach algebras into a Banach bi-module \cite[Theorem
3.3]{Sinclair70}. Nevertheless, a celebrated result of B.E.
Johnson states that every bounded Jordan derivation from a
C$^*$-algebra $A$ to a Banach $A$-bimodule is an associative
derivation (cf. \cite{John96}).

In view of the intense interest in automatic continuity problems
in the past half century, it is natural to ask if the assumption
of boundedness is needed in Johnson's result.  It is therefore
somewhat surprising that the following problem has remained open
for fifteen years.

\begin{problem}\label{prob 1} Is every Jordan derivation from a C$^*$-algebra
$A$ to a Banach $A$-bimodule automatically continuous?
\end{problem}

This problem was already posed in \cite[Question 14.i]{Vill}.
According to \cite[\S 5]{Bre05}, Problem \ref{prob 1} ``is an
intriguing open question''. In 2004, J.
Alaminos\hyphenation{Alaminos}, M. Bre\v{s}ar and A.R. Villena
gave a positive answer to the above problem for some classes of
C$^*$-algebras including the class of von Neumann algebras and the
class of abelian C$^*$-algebras (cf. \cite{AlBreVill}).
In the setting of general C$^*$-algebras the question has remained open.
%We provide a positive answer to this problem in Corollary \ref{cor associative
%derivations from C*-algebras}.\smallskip
\smallskip

Problem~\ref{prob 1} has a natural generalization to the setting
of Banach Jordan algebras. In the category of JB$^*$-algebras, S.
Hejazian and A. Niknam established in \cite{HejNik96} that every
Jordan derivation from a JB$^*$-algebra $J$ into $J$ or into $J^*$
is automatically continuous. We recall that a linear mapping $D$
from a JB$^*$-algebra $J$ to a Jordan Banach $J$-bimodule is said
to be a \emph{Jordan derivation} if $D( a \circ b) = D(a)\circ b +
a \circ D(b)$, for every $a,b$ in $J$, where $\circ$ denotes the
Jordan product in $J$ and the action of $J$ on the Jordan
$J$-module (defined below).\smallskip

The above quoted paper actually contains a theorem which provides
necessary and sufficient conditions to guarantee that a Jordan
derivation from a JB$^*$-algebra $J$ into a Jordan Banach
$J$-module is continuous (cf. \cite[Theorem 2.2]{HejNik96}). The
same authors show the existence of discontinuous Jordan
derivations from JB$^*$-algebras into Jordan Banach modules
(compare \cite[\S 3]{HejNik96}). When the domain JB$^*$-algebra is
a commutative or a compact C$^*$-algebra $A$, the same authors
proved that every Jordan derivation from $A$ into a Jordan Banach
$A$-module is continuous (cf. \cite[Theorem 2.4 and Corollary
2.7]{HejNik96}). In the setting of general C$^*$-algebras,
however, the following question remains open (also for fifteen years).

\begin{problem}\label{prob 2} Is every Jordan derivation from a C$^*$-algebra
$A$ to a Jordan Banach module automatically continuous?
\end{problem}

%We provide a positive answer to this question in Corollary
%\ref{cor Jordan derivations from C*-algebras}.

Prior to the writing of this paper, it apparently had escaped the attention of functional analysts that  combining a theorem of Cuntz (\cite{Cuntz}, see Lemma~\ref{Thm Cuntz} below) with the theorems just quoted from \cite{AlBreVill} and \cite{HejNik96} concerning commutative C$^*$-algebras  yields  positive answers to both Problems \ref{prob 1} and \ref{prob 2}.  We therefore can now state the following theorem.

\begin{theorem}
Every Jordan derivation from a C$^*$-algebra $A$ to a Banach $A$-module or to a Jordan Banach module is continuous.
\end{theorem}

As a consequence of our main results, we are able to treat both
Problems \ref{prob 1} and \ref{prob 2} from a new and more general
point of view. We introduce the class of Banach (Jordan) triple
modules, a class which includes, besides Banach modules over
Banach algebras and Banach Jordan modules over Banach Jordan
algebras, the dual space of every real or complex JB$^*$-triple.
In this setting, a conjugate linear (resp., linear) mapping
$\delta$ from a complex (resp., real) Jordan triple $E$ to a
triple $E$-module is called a \emph{derivation} if
\begin{equation}\label{eq:0409111}
\delta \{a,b,c\} = \J {\delta(a)}bc +\J a{\delta(b)}c + \J
ab{\delta(c)},\end{equation}
 for every $a,b,c\in E$.\smallskip

We determine (in Theorem~\ref{t main complex to module}) the
precise conditions in order that a derivation from a complex
JB$^*$-triple, $E$, into a Banach (Jordan) triple $E$-module is
continuous. We subsequently show that every derivation from a real
or complex JB$^*$-triple into its dual space is automatically
continuous, a fact which has significance for the forthcoming
study by the authors of weak amenability.\smallskip

From one point of view (another is through infinite dimensional
holomorphy) the theory of JB$^*$-triples may be viewed as parallel
to the theory of C$^*$-algebras. The analog of the theorem of
Sakai mentioned above, namely, the  automatic continuity of a
derivation from a JB$^*$-triple into itself, that is, a {\it linear} map satisfying the derivation property (\ref{eq:0409111}),  was proved by T. J.
Barton and Y. Friedman \cite{BarFri} in the complex case and
extended to the real case in \cite{HoMarPeRu}. Among the
consequences of our main results, we obtain a completely different
proof for the automatic continuity results obtained in the just
quoted papers \cite{BarFri} and \cite{HoMarPeRu}.\smallskip

We shall see that there exist examples of triple derivations from
a JB$^*$-triple $E$ to a Banach triple $E$-module which are not
continuous (see Remark \ref{r last}). In our last results we show
that these examples cannot appear when the domain is a
C$^*$-algebra. More concretely, in Theorem \ref{cor ternary
derivations from C*-algebras} and Corollaries \ref{cor Jordan
derivations from C*-algebras}, \ref{cor associative derivations
from C*-algebras}, and \ref{Ringrose} we prove that every triple
(resp., Jordan) derivation from a C$^*$-algebra $A$ to a Banach
triple $A$-module (resp., to a Jordan Banach $A$-module) is
automatically continuous, which constitute the solutions to
Problems 1 and 2 and a completely different proof of the automatic
continuity result of Ringrose quoted above.\smallskip

%Modules over a Jordan pair were introduced in
%\cite{Loos75}.  Derivations on Banach Jordan pairs were studied in
%\cite{Lopezetal01}, in which the automatic continuity of
%derivations on JB$^*$-triples to themselves was also reproved as a
%consequence of its main result. We do not consider Jordan pairs in this paper.
%\smallskip

In section 2 of this paper we recall the definition and basic
properties of Jordan triples, define Jordan triple modules and
submodules, and introduce and study a basic tool in our paper: the
{\it quadratic annihilator} of a submodule. In section 3 we prove
the automatic continuity results by relating triple derivations to
triple module homomorphisms and using the well known technique of
separating spaces. The final section contains an analysis of the
automatic continuity of every triple derivation from a
C$^*$-algebra $A$ to a Banach triple $A$-module, which leads to a unified solution to
Problems \ref{prob 1} and \ref{prob 2}.\smallskip

Our definition of Jordan triple module is motivated by the theory
of modules over a Jordan algebra due to Jacobson \cite{Jac},
together with the definition in the special case of the dual of a
Banach Jordan triple, which was suggested by Tom Barton some time
ago. Subsequently, we noticed that Jordan triple modules were
defined in \cite{Loos1973} in a form more suitable to a purely
algebraic setting. Our definition is more suitable for the
applications to C$^*$-algebras.\smallskip

All of our results, excepting Theorem~\ref{t main complex to
module}, are valid for real or complex JB$^*$-triples. It should
be noted however that the key to the solutions to Problems 1 and 2
is that Theorem~\ref{t main complex to module} is valid for the
self-adjoint part of a C$^*$-algebra, considered as a (reduced)
real JB$^*$-triple (see Proposition~\ref{thm PeRu real}).

\section{Jordan triple Modules}

\subsection{Jordan triples}

A complex (resp., real) \emph{Jordan triple} is a complex (resp., real)
vector space $E$ equipped with a triple
product $$ E \times E \times E \rightarrow E$$
$$(xyz) \mapsto \J xyz $$
which is bilinear and symmetric in the outer variables and
conjugate linear (resp., linear) in the middle one satisfying the
so-called \emph{``Jordan Identity''}:
$$L(a,b) L(x,y) -  L(x,y) L(a,b) = L(L(a,b)x,y) - L(x,L(b,a)y),$$
for all $a,b,x,y$ in $E$, where $L(x,y) z := \J xyz$. When $E$ is
a normed space and the triple product of $E$ is continuous, we say
that $E$ is a \emph{normed Jordan triple}. If a normed Jordan
triple $E$  is complete with respect to the norm (i.e. if $E$ is a
Banach space), then it is called a
\emph{Jordan-Banach triple}.
Unless otherwise specified, the term ``normed Jordan triple''
(resp., ``Jordan-Banach triple'') will always mean a real or
complex normed Jordan triple (resp., a real or complex
Jordan-Banach triple).\smallskip

%For each Jordan-Banach triple $E$, the constant $N(E)$ will denote
%the supremum of the set $\{\|\J xyz \|: \|x\|,\|y\|,\|z\|\leq
%1\}$.\smallskip

A summary of the basic facts about the important subclass of JB$^*$-triples (defined below), some of which are recalled here, can be
found in \cite{Russo94} and some of the references therein, such
as \cite{Ka},\cite{FriRu85},\cite{FriRus86bis},\cite{U1} and \cite{U2}.\smallskip

A subspace $F$ of a Jordan triple $E$ is said to be a
\emph{subtriple} if $\J FFF \subseteq F$. We recall that a
subspace $J$ of $E$ is said to be a \emph{triple ideal} if
$\{E,E,J\}+\{E,J,E\} \subseteq J.$ When $\J JEJ \subset J$ we say
that $J$ is an \emph{inner ideal} of $E$.\smallskip

We recall that a  real (resp., complex) \emph{Jordan algebra} is a
(not-necessarily associative) algebra over the real (resp.,
complex) field whose product is abelian and satisfies $(a \circ
b)\circ a^2 = a\circ (b \circ a^2)$. A normed Jordan algebra is a
Jordan algebra $A$ equipped with a norm, $\|.\|$, satisfying $\|
a\circ b\| \leq \|a\| \ \|b\|$, $a,b\in A$. A \emph{Jordan Banach
algebra} is a normed Jordan algebra whose norm is
complete.\smallskip

Every Jordan algebra is a Jordan triple with respect to $$\J abc := (a\circ b) \circ c + (c\circ b) \circ a - (a\circ c) \circ b.$$

Every real or complex associative Banach algebra (resp., Jordan
Banach algebra) is a real Jordan-Banach triple with respect to the
product $\J abc = \frac12 (a bc +cba)$ (resp., $\J abc = (a\circ
b) \circ c + (c\circ b) \circ a - (a\circ c) \circ b$).\smallskip

An element $e$ in a Jordan triple $E$ is called \emph{tripotent}
if $\J eee =e$. Each tripotent $e$ in $E$ induces two
decomposition of $E$ (called \emph{Peirce decompositions}) in the
form:
$$E=E_0(e)\oplus E_1(e)\oplus E_2(e)=E^1 (e) \oplus E^{-1} (e)
\oplus E^{0} (e) $$ where $E_k(e)=\{x\in E:L(e,e)x=\frac k2 x\}$
for $k=0,1,2$ and $E^{k} (e)$ is the $k$-eigenspace of the
operator $Q(e) x=\J exe$ for $k=1,-1,0$.  The projection onto $E_k(e)$, which is contractive, is denoted by $P_k(e)$ for $k=0,1,2$. The following
\emph{Peirce rules} are satisfied:
\begin{enumerate}[$(a)$]
\item $E_{2} (e) = E^1 (e) \oplus E^{-1} (e)$ and $E^{0} (e) =
E_{1} (e) \oplus E_{0} (e)$,\smallskip

\item $\{ E^{i} (e), E^{j} (e), E^{k} (e)\} \subseteq E^{i j k}
(e)$ if $i j k \neq 0$,\smallskip

\item  $\{E_{i_{}}(e),E_{j_{}}(e),E_{k_{}}(e)\}\subseteq
E_{i-j+k}(e),$ where $i,j,k=0,1,2$ and $E_{l_{}}(e)=0$ for $l\neq
0,1,2$,\smallskip

\item  $\{E_0(e),E_2(e),E\}=\{E_2(e),E_0(e),E\}=0.$\smallskip
\end{enumerate}

A JB$^*$-algebra is a complex Jordan Banach algebra $A$ equipped
with an algebra involution $^*$ satisfying  $\|\J a{a^*}a \|= \|a\|^3$, $a\in
A$.  (Recall that $\J a{a^*}a  =
 2 (a\circ a^*) \circ a - a^2 \circ a^*$).\smallskip

A \emph{(complex) JB$^*$-triple} is a complex Jordan Banach triple
${E}$ satisfying the following axioms: \begin{enumerate}[($JB^*
1$)] \item For each $a$ in ${E}$ the map $L(a,a)$ is an hermitian
operator on $E$ with non negative spectrum. \item  $\left\|
\{a,a,a\}\right\| =\left\| a\right\| ^3$ for all $a$ in ${A}.$
\end{enumerate}\smallskip

Every C$^*$-algebra (resp., every JB$^*$-algebra) is a JB$^*$-triple with respect to the product
$\J abc = \frac12 \ ( a b^* c + cb^* a) $ (resp., $\J abc := (a\circ b^*) \circ c + (c\circ b^*) \circ a - (a\circ c) \circ b^*$).\smallskip

We recall that a \emph{real JB$^*$-triple} is a norm-closed real
subtriple of a complex JB$^*$-triple (compare
\cite{IsKaRo95}). The class of real JB$^*$-triples includes all complex JB$^*$-triples,
all real and complex C$^*$- and JB$^*$-algebras and all JB-algebras.\smallskip

A complex (resp., real) \emph{JBW$^*$-triple} is a complex (resp.,
real) JB$^*$-triple which is also a dual Banach space (with a
unique isometric predual \cite{BarTi,MarPe}). It is known that the
triple product of a JBW$^*$-triple is separately weak$^*$
continuous (c.f. \cite{BarTi} and \cite{MarPe}). The second dual
of a JB$^*$-triple $E$ is a JBW$^*$-triple with a product
extending the product of $E$ \cite{Di86b,IsKaRo95}.\smallskip

It is also known that, for each tripotent $e$ in a complex
JB$^*$-triple $E$, $E_2 (e)$ is a JB$^*$-algebra with product and
involution given by $x\circ_{_e} y := \J xey$ and $x^{\sharp_{_e}}
:= \J exe$, respectively. In the case of $E$ being a real
JB$^*$-triple $E^{1} (e)$ is a JB-algebra with respect to the
product given in the above lines (JB-algebras are precisely the
self adjoint parts of JB$^*$-algebras \cite{Wright77}).\smallskip

A tripotent $e$ in a real or complex JB$^*$-triple $E$ is called
\emph{minimal} if $E^{1} (e) = \RR e$. In the complex setting this
is equivalent to say that $E_2 (e) = \CC e$, because $E^{-1} (e) =
i E^{1} (e)$, whereas in the real situation the dimensions of
$E^{1} (e)$ and $E^{-1} (e)$ need not be correlated.\smallskip

When $E$ is a JB$^*$-triple or a real JB$^*$-triple, a subtriple $I$
of $E$ is a triple ideal if and only if $\J EEI \subseteq I$ or $\J EIE \subseteq I$
 or $\J EII\subseteq I$ (compare \cite{Bun86}).\smallskip

\subsection{Jordan triple modules}
Let $A$ be an associative algebra. Let us recall that an
\emph{$A$-bimodule} is a vector space $X$, equipped with two
bilinear products $(a,x)\mapsto a x$ and $(a,x)\mapsto x a$ from
$A\times X$ to $X$ satisfying the following axioms: $$a (b x) = (a
b) x ,\ \ a (x b) = (a x) b, \hbox{ and, }(xa) b = x (a b),$$ for
every $a,b\in A$ and $x\in X$.\smallskip

Let $J$ be a Jordan algebra. A \emph{Jordan $J$-module} is a
vector space $X$, equipped with two bilinear products
$(a,x)\mapsto a \circ x$ and $(x,a)\mapsto x \circ a$ from
$J\times X$ to $X$, satisfying: $$a \circ x = x\circ a,\ \ a^2
\circ (x \circ a) = (a^2\circ  x)\circ a, \hbox{ and, }$$ $$2(
(x\circ a)\circ  b) \circ a + x\circ (a^2 \circ b) = 2 (x\circ
a)\circ  (a\circ b) + (x\circ b)\circ a^2,$$ for every $a,b\in J$
and $x\in X$ (see  \cite[\S II.5,p.82]{Jac}).
\smallskip

Let $E$ be a complex (resp. real) Jordan triple.  A \emph{Jordan
triple $E$-module}  (also called \emph{triple $E$-module}) is a
vector space $X$ equipped with three mappings $$\{.,.,.\}_1 :
X\times E\times E \to X, \quad \{.,.,.\}_2 : E\times X\times E \to
X$$ $$ \hbox{ and } \{.,.,.\}_3: E\times E\times X \to X$$
satisfying  the following axioms:
\begin{enumerate}[{$(JTM1)$}]
\item $\{ x,a,b \}_1$ is linear in $a$ and $x$ and conjugate
linear in $b$ (resp., trilinear), $\{ abx \}_3$ is linear in $b$
and $x$ and conjugate linear in $a$ (resp., trilinear) and
$\{a,x,b\}_2$ is conjugate linear in $a,b,x$ (resp., trilinear)
\item  $\{ x,b,a \}_1 = \{ a,b,x \}_3$, and $\{ a,x,b \}_2  = \{
b,x,a \}_2$  for every $a,b\in E$ and $x\in X$. \item Denoting by
$\J ...$ any of the products $\{ .,.,. \}_1$, $\{ .,.,. \}_2$ and
$\{ .,.,. \}_3$, the identity $\J {a}{b}{\J cde} = \J{\J abc}de $
$- \J c{\J bad}e +\J cd{\J abe},$ holds whenever one of the
elements  $a,b,c,d,e$ is in $X$ and the rest are in $E$.
\end{enumerate}

When $E$ is a Jordan Banach triple and $X$ is a triple $E$-module
which is  also a Banach space and, for each $a,b$ in $E$, the
mappings $x\mapsto \J abx_3$ and $x\mapsto \J axb_2$ are
continuous, we shall say that $X$ is a triple $E$-module with
\emph{continuous module operations}. When the products $\{ .,.,.
\}_1$, $\{ .,.,. \}_2$ and $\{ .,.,. \}_3$ are (jointly)
continuous we shall say that $X$ is a \emph{Banach (Jordan) triple
$E$-module}. \smallskip

It is obvious that every real or complex Jordan triple $E$ is a
{\it real}  triple $E$-module. Actually, every triple ideal $J$ of
$E$ is a (real) triple $E$-module. It is problematical whether
every complex Jordan triple $E$ is a complex triple $E$-module for
a suitable triple product. We shall see later that triple modules
have a priori a different behavior than bi-modules over
associative algebras and Jordan modules (see Remark \ref{r last}).\smallskip

Every real or complex associative algebra $A$ (resp.,  Jordan
algebra $J$) is a real Jordan  triple with respect to $\J abc :=
\frac12 \left(abc +cba\right)$, $a,b,c\in A$ (resp., $\J abc =
(a\circ b) \circ c + (c\circ b) \circ a - (a\circ c) \circ b$) ,
$a,b,c\in J$). It is not hard to see that every $A$-bimodule $X$
is a real triple $A$-module with respect to the products $\J abx_3
:=   \frac12 \left(abx +xba\right )$ and  $\J axb_2 = \frac12
\left(axb +bxa\right)$, and that every Jordan module $X$ over a
Jordan algebra $J$ is a real triple $J$-module with respect to the
products $\J abx_3 :=  (a\circ b) \circ x + (x\circ b) \circ a -
(a\circ x) \circ b$ and  $\J axb_2 = (a\circ x) \circ b + (b\circ
x) \circ a - (a\circ b) \circ x.$
\smallskip

Hereafter, the triple products $\J \cdot\cdot\cdot_j$, $ j=1,2,3$, which occur in the definition of Jordan triple module
will be denoted simply by $\J \cdot\cdot\cdot$ whenever the
meaning is clear from the context.
\smallskip

It is a little bit more laborious to check that the dual space,
$E^*$, of a complex (resp., real) Jordan Banach triple $E$ is  a
complex (resp., real) triple $E$-module with respect to the
products: \begin{equation}\label{eq module product dual 1}   \J
ab{\varphi} (x) = \J {\varphi}ba (x) := \varphi \J bax
\end{equation} and \begin{equation}\label{eq module product dual
2} \J a{\varphi}b (x) := \overline{ \varphi \J axb }, \forall \varphi\in
E^*, a,b,x\in E. \end{equation}

Given a triple $E$-module $X$  over a Jordan triple $E$, the
space $E\oplus X$ can be equipped with a structure of real Jordan
triple with respect to the product $\J {a_1+x_1}{a_2+x_2}{a_3+x_3}
= \J {a_1}{a_2}{a_3} +\J  {x_1}{a_2}{a_3}+\J  {a_1}{x_2}{a_3} + \J
{a_1}{a_2}{x_3}$. Consistent with the terminology in \cite[\S
II.5]{Jac}, $E\oplus X$ will be called the \emph{triple split null extension}
of $E$ and $X$. \smallskip

A subspace $S$ of a triple $E$-module $X$ is said to be a
\emph{Jordan triple submodule} or a \emph{triple submodule} if and
only if $\J EES \subseteq S$ and $ \J ESE \subseteq S$. Every triple ideal $J$ of
$E$ is a Jordan triple $E$-submodule of $E$.\smallskip

\subsection{Quadratic annihilator}
Given an element $a$ in a Jordan triple $E$, we shall denote by
$Q(a)$ the conjugate linear operator on $E$ defined by $Q(a) (b)
:= \J aba.$ The  following formula is always satisfied
\[
Q(a) Q(b) Q(a) = Q(Q(a)b), \ \ (a,b\in E).\]
and remains true for $Q(\cdot)$ acting on a triple $E$-module $X$:
\begin{equation}\label{basic equation}
\J{a}{\J{b}{\J  axa}{b}}{a}=\J{\J aba}{x}{\J aba}\ ,\  x\in X.
\end{equation}

For each submodule $S$ of a  triple $E$-module $X$, we define its
\emph{quadratic annihilator}, Ann$_{E} (S)$,  as the set $\{ a\in
E : Q (a) (S) = \J aSa = 0\}$. Since $S$ is triple submodule of
$X$, it follows by $(\ref{basic equation})$ that \begin{equation}
\label{eq submodule 1} \J {a}E {a} \subset
{\hbox{Ann}_{E} (S)}, \ \forall a\in \hbox{Ann}_{E} (S)
\end{equation}
and
\begin{equation}
\label{eq submodule 2} \J b{\hbox{Ann}_{E} (S)}b \subseteq
{\hbox{Ann}_{E} (S)}, \ \forall b\in E.
\end{equation}

Consequently, $ {\hbox{Ann}_{E} (S)}$ is an inner ideal of $E$
whenever it is a linear subspace of $E$. Further,  $
{\hbox{Ann}_{E} (S)}$ is  a triple ideal of $E$ whenever $E$ is a
JB$^*$-triple  and $ {\hbox{Ann}_{E} (S)}$ is a linear subspace of
$E$, since as noted earlier, for JB$^*$-triples,  (\ref{eq
submodule 2}) implies $\J{E}{\hbox{Ann}_{E} (S)}{E}\subset
\hbox{Ann}_{E} (S)$.\label{ref ann is an ideal}\smallskip

Let $E$ be a Jordan triple. Two elements $a$ and $b$ in $E$ are
said to be \emph{orthogonal} (written $a\perp b$) if $L(a,b) =
L(b,a)=0$. A direct application of the Jordan identity yields
that, for each $c$ in $E$,
\begin{equation}\label{eq orth inner ideal} a\perp \J bcb  \hbox{ whenever } a\perp b.\end{equation}

Given an element $a$ in a  Jordan triple $E$, we denote $a^{[1]} =
a$, $a^{[3]} = \J aaa$ and $a^{[2 n +1]} := \J a{a^{[2n-1]}}a$
$(\forall n\in \mathbb{N})$. The Jordan identity implies that
$a^{[5]} = \J aa{a^{[3]}}$ , and by induction, $a^{[2n+1]} =
L(a,a)^n (a)$ for all $n\in\NN$. The element $a$ is called
\emph{nilpotent} if $a^{[2n+1]}=0$ for some $n$. Jordan triples
are power associative, that is,
$\J{a^{[k]}}{a^{[l]}}{a^{[m]}}=a^{[k+l+m]}$.\smallskip

A Jordan triple $E$ for which the vanishing of
$\J aaa$ implies that $a$ itself vanishes is said to be
\emph{anisotropic}. It is easy to check that $E$ is anisotropic
if and only if zero is the unique nilpotent element in $E$.\smallskip

Let $a$ and $b$ be two elements in a Jordan
triple $E$. If $L(a,b)=0$, then, for each $c$ in $E$, the Jordan
identity implies that $$\J {L(b,a)c}{L(b,a)c}{L(b,a)c} = 0.$$
Therefore, in an anisotropic Jordan triple, $a\perp b$ if and only if
$L(a,b)=0$.\smallskip

Let $a$ be an
element in a real (resp., complex) JB$^*$-triple $E$. Denoting by $E_a$ the
JB$^*$-subtriple generated by the element $a$, it is known that $E_a$
is JB$^*$-triple isomorphic (and hence isometric) to $C_0 (L)=
C_0(L,\RR)$ (resp., $C_0 (L)=
C_0(L,\CC)$) for some locally compact Hausdorff space $L\subseteq
(0,\|a\|],$ such that $L\cup \{0\}$ is compact. It is also known
that denoting by $\Psi$ the triple isomorphism from $E_a$ onto
$C_{0}(L),$ then $\Psi (a) (t) = t$ $(t\in L)$ (compare
\cite[Lemma 1.14]{Ka}, \cite[Proposition 3.5]{Ka96} or \cite[Page 14]{BurPeRaRu}).
The set $L$ is called the \emph{triple spectrum} of $a$.\smallskip

It should be noticed here that, in the setting of real or complex JB$^*$-triples
 orthogonality is a ``local concept''(compare Lemma 1 in \cite{BurFerGarMarPe},
 whose proof remains
valid for real JB$^*$-triples). Indeed, two elements $a$ and $b$ in a real
JB$^*$-triple $E$ are orthogonal if and only if one of the
following equivalent statements holds: $$ (a) \ \J aab =0, \ \ (b)
\ E_a \perp E_b, \ \ (c) \ \{b,b,a\} =0,$$
$$(d)\  a\perp b \hbox{ in a subtriple of $E$ containing both
elements}.$$\vspace{0,5mm}

Let $E$ be a (real or complex) Jordan Banach triple. We have
already mentioned that $E^*$ is a triple module with respect to the
products given in $(\ref{eq module product dual 1})$ and $(\ref{eq
module product dual 2})$. The triple module structure of $E^*$
satisfies the following additional property: given $a$ and $b$ in
$E$ with $a\perp b$ (in $E$), we have $\J ab{\varphi} =  \J
{\varphi}ba = 0$ for every $\varphi \in E^*$. That is, $a\perp b$
in the Jordan triple $E\oplus E^*$. Orthogonal elements in $E$
lift to orthogonal elements in the split null extension $E\oplus
E^*$.\smallskip

Let $X$ be a triple module over a Jordan triple $E$. We shall say
that $X$ has the property of \emph{lifting orthogonality} (LOP in
short) if $$\J abx = 0,\hbox{ for every } x\in X, \ a,b \in E
\hbox{ with } a\perp b.$$

We have just remarked that for every Jordan Banach triple $E$,
$E^*$ is a triple $E$-module satisfying LOP. When a Jordan triple
$E$ is regarded as a real triple $E$-module with its natural
products, then $E$ also has LOP. However, not every triple module
has this property. Let $A$ be a C$^*$-algebra regarded as a
complex JB$^*$-triple  with respect to $\J abc := \frac12 (a b^* c
+c b^* a)$. As noted earlier, the vector space $X=A$ is a real triple $A$-module
with respect to the products $\{a,b,x\}_3 := \frac12 (a b x +x b a)$
and $\{a,x,b\}_2 := \frac12 (a x b+ b x a)$. Two elements $a$ and $b$
in a real or complex  C$^*$-algebra $A$ are orthogonal  if and only if $ ab^* = b^* a = 0$ or equivalently,
in the
triple sense,
$aa^*b + b a^* a = 0$ or  $bb^*a + a b^* b = 0$ (compare
\cite[Lemma 1]{BurFerGarMarPe}). It is not hard to find a
C$^*$-algebra $A$ containing two orthogonal elements $a,b$ with
$\{a,b,x\}_3 \neq 0$ for some $x\in A$.\smallskip

Let $J$ be a norm-closed subspace of a JB$^*$-triple (resp., a
real JB$^*$-triple) $E$. Clearly, $J$ is a triple ideal of $E$
if and only if $J$ is a triple $E$-submodule of $E$.
Let $J$ be a triple ideal of $E$ regarded as a Jordan triple
$E$-submodule. We clearly have $$\hbox{Ann}_{_{E}} (J) : =
\{ a\in E : Q(a) (J) = 0\} \supseteq J^{\perp}:= \{ a \in E :
a\perp J\}.$$ Suppose now that $a\in \hbox{Ann}_{_{E}} (J)$.
Replacing $J$ with its weak$^*$-closure in $E^{**}$, we may assume
that $E$ is a JBW$^*$-triple, $J$ is a weak$^*$-closed triple
ideal and $Q(a) (J) =0.$ By \cite[Theorem 4.2 (4)]{Ho87}, there
exists a weak$^*$-closed triple ideal $K$ in $E$ such that $E =
J\oplus K$ and $J\perp K$. Writing $a= a_1+a_2$ with $a_1\in J$ and
$a_2\in K$, we deduce, by orthogonality, that $a_1^{[3]} = Q(a)
(a_1) \in Q(a) (J) = 0$, and hence $a= a_2\perp J$. We state this as a Lemma.\smallskip

\begin{lemma}
\label{l annihilator triple ideal} Let $E$ be a JB$^*$-triple
(resp., a real JB$^*$-triple). For each triple ideal $J$ in $E$ we
have $\hbox{Ann}_{_{E}} (J) = J^{\perp}$ is a norm closed triple
ideal of $E$.$\hfill\Box$
\end{lemma}

Let $E$ be a JB$^*$-triple (resp., a real JB$^*$-triple). For each
$x$ in $E$, $E(x)$ will denote the norm-closure of $\J xEx$ in
$E$. It is known that  $E(x)$ coincides with the norm-closed inner
ideal of $E$ generated by $x$  and $E_x \subseteq E(x)$ (see
\cite{BuChuZa1}). By \cite[Proposition 2.1]{BuChuZa1}, $E(x)$ is a
JB$^*$-subalgebra of the JBW$^*$-algebra $E(x)^{**} =
\overline{E(x)}^{w^*} = E^{**}_{2} (r(x)),$ where $r(x)$ is the (so called)
range tripotent of $x$ in $E^{**}$. It is also known that $x\in
E(x)_{+}$.\label{inner ideal}\smallskip

For each functional $\varphi\in E^*$, there exists a unique
tripotent $s=s(\varphi)$ in $E^{**}$ satisfying that $\varphi =
\varphi P_2 (s)$ and $\varphi|_{E^{**}_2 (s)}$ is a faithful
normal positive functional on $E^{**}_2 (s)$ (compare
\cite[Proposition 2]{FriRu85} and \cite[Lemma 2.9]{MarPe} and
\cite[Lemma 2.7]{PeSta}, respectively). The tripotent $s(\varphi)$
is called the \emph{support tripotent} of $\varphi$ in $E^{**}$.

\begin{proposition}
\label{p annihilator submodule dual} Let $E$ be a JB$^*$-triple
(resp., a real JB$^*$-triple). For each triple submodule $S\subset
E^*$,
\begin{itemize}
\item[$(a)$] the quadratic annihilator $\hbox{Ann}_{_{E}} (S)$ is
a norm closed triple ideal of $E$, \item[$(b)$] $\hbox{Ann}_{_{E}}
(S)= E\cap \left(\bigcap_{\varphi\in S}  E^{**}_0
(s(\varphi))\right)$, \item[$(c)$] $\J {\hbox{Ann}_{_{E}}
(S)}{\hbox{Ann}_{_{E}} (S)}{S} = 0$  in the triple split null
extension $E\oplus E^*$.
\end{itemize}
\end{proposition}

\begin{proof} We prove (b) first.
For each $a\in \hbox{Ann}_{_{E}} (S)$ and each $\varphi \in S$, we
have by definition,  $\J a{\varphi}a =0$ and hence $\varphi Q(a)
(E) =0$. It follows that $E(a) \subseteq \ker(\varphi)$ for every
$\varphi\in S$, $a\in \hbox{Ann}_{_{E}} (S)$. In particular,
$\varphi (a) =0$. Since $S$  is a triple submodule, for every
$b\in E$, $\{\varphi,b,a\}\in S$, so
$\{\varphi,b,a\} (a) = 0$, that is,  $\varphi\J aab = 0$.\smallskip

Fix $\varphi\in S$. We have already seen that $\varphi\J aab = 0$
for every $b\in E$. Since $E$ is weak$^*$-dense in $E^{**}$ and
$\varphi\J aa. $ is weak$^*$-continuous on $E^{**}$, we deduce
that $\varphi\J aab = 0$, for every $b\in E^{**}$. Thus,
\begin{equation} \label{eq proof 1} \varphi\J aa{s(\varphi)} = 0,
\end{equation} where $s=s(\varphi)\in E^{**}$ denotes the  support tripotent
of $\varphi$ in $E^{**}$.\smallskip

Proposition 2 and Lemma 1.5 in \cite{FriRu85} together with the
Peirce arithmetic imply that the mapping $$(x,y) \mapsto \varphi\J
xy{s}= \varphi\J {P_2(s)x}{P_2(s)y}{s}+ \varphi\J
{P_1(s)x}{P_1(s)y}{s}$$ is faithful and  positive on $E^{**}_2
(s)\oplus E^{**}_1 (s)$, that is, $\varphi\J xx{s} \geq 0$ for
every $x \in E^{**}_2 (s)\oplus E^{**}_1 (s)$ and $\varphi\J xx{s}
= 0$ if and only if $x =0$. By $(\ref{eq proof 1})$, $$0=\varphi\J
aa{s(\varphi)} =\varphi\J {P_2(s)a+P_1(s) a}{P_2(s)a+P_1(s)
a}{s},$$ which implies that $P_2(s)a=P_1(s) a=0$.\smallskip

We have shown that $\hbox{Ann}_{_{E}} (S) \subseteq E\cap E^{**}_0
(s(\varphi))$, for every $\varphi \in S$.  This assures that
\begin{equation}\label{eq:0531101}
\hbox{Ann}_{_{E}} (S)\subseteq  E\cap \left(\bigcap_{\varphi\in S}  E^{**}_0 (s(\varphi))\right).
\end{equation}
To prove the reverse inclusion, let $b$ belong to the right side
of (\ref{eq:0531101}), let $\varphi\in S$ and let $c\in E$ have
Peirce decomposition $c=c_2+c_1+c_0$ with respect to $s(\varphi)$.
From Peirce arithmetic, $\J{b}{\varphi}{b}(c)=\varphi\J bcb
=\varphi\J {b}{c_0}{b}=0$, proving equality in (\ref{eq:0531101})
and establishing (b).\smallskip

To prove (c), let $b,c\in \hbox{Ann}_{_{E}} (S)$ and $\varphi\in S$. Then for
$x=x_2+x_1+x_0 \in E$ (with respect to $s(\varphi)$),  by Peirce
rules and properties of the support tripotent,  $\J bc\varphi
(x)=\varphi \J cbx=\varphi\J {c}{b}{x_2}+\varphi\J
{c}{b}{x_1}+\varphi\J {c}{b}{x_0}=0$, which proves c).\smallskip

Because of (\ref{eq submodule 1}) and (\ref{eq submodule 2}), to
prove (a) it remains to show that  $\hbox{Ann}_{_{E}} (S)$ is a
linear subspace of $E$. Take $a,b\in \hbox{Ann}_{_{E}} (S)$.
Since, by Peirce arithmetic, $Q(a,b) (E) \subseteq E\cap E^{**}_0
(s(\varphi)),$ and $L(a,b) (E) \subseteq E\cap \left( E^{**}_0
(s(\varphi)) \oplus E^{**}_1 (s(\varphi))\right),$ for every
$\varphi\in S$, it follows that $\J a{\varphi}b =0,$ and  $\J
ab{\varphi} =0,$ for every $\varphi\in S$. Therefore  (using only the first of these two facts),
$$ Q(a+b) \varphi= Q(a) \varphi+ Q(b) \varphi+ 2 Q(a,b) \varphi=0,$$
for every $a,b\in \hbox{Ann}_{_{E}} (S)$ and $\varphi \in S,$
which implies that $\hbox{Ann}_{_{E}} (S)$ is a linear subspace of
$E$ and completes  the proof.
\end{proof}

\begin{remark}\label{r real and complex JB$^*$-triples OAP}
Let $E$ be a real or complex JB$^*$-triple regarded as a real
Banach triple $E$-module. It can be easily seen that norm-closed
triple $E$-submodules and norm-closed triple ideals of $E$
coincide. The conclusions in the above Proposition \ref{p
annihilator submodule dual} remain true for any norm-closed triple
$E$-submodule (i.e. norm-closed triple ideal) of $E$. Indeed, let
$S=J$ be a norm-closed triple ideal of $E$. By Lemma \ref{l
annihilator triple ideal}, $\hbox{Ann}_{E} (J) = J^{\perp}$, which
implies that  $$\J {\hbox{Ann}_{_{E}} (J)}{\hbox{Ann}_{_{E}}
(J)}{J} = 0,$$ in the triple split null extension $E\oplus E$.
\end{remark}

\section{Triple derivations and triple module homomorphisms}

\subsection{Triple derivations}

Separating spaces have been revealed as a useful tool
in results of automatic continuity. This tool has been applied
by many authors in the  study of automatic continuity of binary and ternary
homomorphims,  derivations and module homomorphisms (see,
for example,  \cite{Rick50,BaCur,Yood,JohnSin68,JohnSin69,Sin74,Sin76,Dales78}
and \cite{Dales89}, among others).
These spaces also play an important role in the subsequent generalisations
of Kaplansky's theorem (compare \cite{CLev,HejNik} and \cite{FerGarPe}).\smallskip

Let $T: X\to Y$ be a linear mapping between two
normed spaces. Following \cite[Page 70]{Rick}, the
\emph{separating space}, $\sigma_{_Y} (T)$, of $T$ in $Y$ is defined
as the set of all $z$ in $Y$ for which there exists a sequence
$(x_n) \subseteq X$ with $x_n \rightarrow 0$ and
$T(x_n)\rightarrow z$. The \emph{separating space}, $\sigma_{_X}
(T)$, of $T$ in $X$ is defined by $\sigma_{_X} (T):=T^{-1}(\sigma_{_Y}
(T)).$

A straightforward application of the closed graph theorem shows
that a linear mapping $T$ between two Banach spaces $X$ and $Y$ is
continuous if and only if $\sigma_{_Y} (T) =\{0\}$ (c.f.
\cite[Proposition 4.5]{CLev}).
It is known that  $\sigma_{_X} (T)$ and $\sigma_{_Y} (T)$ are closed linear
subspaces of $X$ and $Y,$ respectively. \smallskip

A useful property of the separating space
$\sigma_{_{Y}} (T)$ asserts that for every bounded linear operator
$R$ from $Y$ to another Banach space $Z$, the composition $R T$ is
continuous if and only if $\sigma_{_Y}(T)\subseteq \ker (R)$.
Further, there exists a constant $M>0$ (which does not depend on
$R$ nor $Z$) such that $\|RT \| \leq M \ \|R\|$, whenever $RT$ is
continuous (compare \cite[Lemma 1.3]{Sin76}).\smallskip

Let $E$ be a complex (resp., real) Jordan triple and let $X$ be a
triple $E$-module. We recall that a conjugate linear (resp.,
linear) mapping $\delta : E \to X$ is said to be a
\emph{derivation} if
$$\delta \{a,b,c\} = \J {\delta(a)}bc +\J a{\delta(b)}c + \J ab{\delta(c)}.$$

Note that derivations on complex JB$^*$-triples to themselves are
linear mappings but that a derivation from a complex JB$^*$-triple
into a complex triple module is conjugate linear  by this
definition. This is not inconsistent, since as we have noted
earlier, it is not clear that a complex JB$^*$-triple $E$ can be
made into a complex triple $E$-module.

\begin{lemma}\label{l separating spaces of a derivation} Let
$\delta : E \to X$ be a triple derivation from a Jordan Banach
triple to a Banach (Jordan) triple $E$-module. Then $\sigma_{_X}
(\delta)$ is a norm-closed triple $E$-submodule of $X$ and
$\sigma_{_E} (\delta)$ is a norm-closed subtriple of $E$.
\end{lemma}

\begin{proof} Given $a,b$ in $E$ and $x\in \sigma_{_X} (\delta)$,
there exists a sequence $(c_n)$ in $E$ with $(c_n)\to 0$ and
$\delta (c_n) \to x$ in norm. The sequence $(\J ab{c_n})$ (resp.,
$(\J a{c_n}b)$) tends to zero in norm and $\delta \J ab{c_n} = \J
{\delta{a}}b{c_n} + \J a{\delta b}{c_n} + \J ab{\delta(c_n)}\to \J
ab{x}$ (resp., $\delta \J a{c_n}b \to \J axb$), which proves the
first statement.\smallskip

If $a,b,c\in \sigma_E(\delta)$, then $\delta a,\delta b,\delta c\in \sigma_X(\delta)$ and by the first statement
 $\delta\J{a}{b}{c}\in\sigma_X(\delta)$, as required.
\end{proof}

Let $\delta : E \to X$ be a triple derivation from a Jordan Banach
triple $E$ to a Banach triple $E$-module. Since $\sigma_{_X}
(\delta)$ is a norm closed triple $E$-submodule of $X$,  Ann$_{E}
(\sigma_{_X} (\delta))$ is a norm closed inner ideal of $E$
whenever it is a linear subspace of $E$ (actually, in such a case,
it is a triple ideal when $E$ is a real or complex
JB$^*$-triple).\smallskip

 Let us take $a$ in $E$. Since $\delta$ is in particular a linear mapping, from the useful property mentioned above,
$\sigma_{_X} (\delta) \subseteq \ker (Q(a))$ if and only if $Q(a)
\delta$ is a continuous linear mapping from $E$ to $X$, and we
deduce that $$\hbox{Ann}_{E} (\sigma_{_X} (\delta)) =\{ a\in E :
Q(a) \delta \hbox{ is continuous}\}.$$ Moreover, for each $a$ in
$E$, $\delta Q(a) = Q(a) \delta + 2Q(a,\delta a)$, and  it follows
that  $Q(a) \delta$ is continuous if and only if $\delta Q(a)$ is.

\subsection{Triple module homomorphisms}

Let $X$ and $Y$ be two triple $E$-modules over a real or complex
Jordan triple $E$.  A linear mapping \linebreak $T : X \to Y$ is
said to be a \emph{triple $E$-module homomorphism} if the
identities $$T \J abx = \J ab{T(x)} \hbox{ and }  T \J axb = \J
a{T(x)}b,$$ hold for every $a,b\in E$ and $x\in X$. \smallskip

As above,
$$\hbox{Ann}_{E} (\sigma_{_Y} (T)) =\{ a\in E :
Q(a) T \hbox{ is continuous}\}$$
and
since a triple module
$E$-homomorphism $T: X \to Y$  commutes with $Q(a)$ (acting on $X$), we have
 $$\hbox{Ann}_{E} (\sigma_{_Y} (T)) =\{ a\in E :  T Q(a) \hbox{ is continuous}\},$$
where $Q(a)$ acts on $Y$.
\smallskip

The argument applied in the proof of Lemma \ref{l separating
spaces of a derivation} is also valid to prove the following
result.

\begin{lemma}\label{l separating spaces of a module homomorphism}
Let $E$ be a Jordan Banach triple and let  $T : X \to Y$ be a
triple $E$-module homomorphism between two Banach space which are
triple $E$-modules with continuous module operations. Then
$\sigma_{_Y} (T)$ and $\sigma_{_X} (T)$ are  norm closed triple
$E$-submodules of $Y$ and $X$, respectively.$\hfill\Box$
\end{lemma}

%Given $a,b$ in $E$ and $y\in \sigma_{Y} (T)$, there exists a sequence $(x_n)$ in $X$ with $(x_n)\to 0$ and %$T (x_n) \to y$ in norm. The sequence $(\J ab{x_n})$ (resp., $(\J a{x_n}b)$) tends to zero in norm in $X$ and %$T \J ab{x_n} =  \J ab{T(x_n)}\to \J ab{y}$ (resp., $T \J a{a_n}b \to \J ayb$), which proves the first statement. %
%Given $a,b$ in $E$ and $x\in \sigma_{X} (T)$, there exists a sequence $(x_n)$ in $X$ with $(x_n)\to 0$ and %$T (x_n) \to T(x)$ in norm. The sequence $(\J ab{x_n})$ (resp., $(\J a{x_n}b)$) tends to zero in norm in $X$ %and $T \J ab{x_n} =  \J ab{T(x_n)}\to \J ab{T(x)}$ (resp., $T \J a{a_n}b \to \J a{T(x)}b$), which proves the first %statement.

The following lemma provides a key tool needed in our main result.

\begin{lemma}\label{l 0} Let $E$ be a Jordan Banach triple,
$X$ a Banach triple $E$-module satisfying LOP, $Y$ a Banach space
which is a triple $E$-module with continuous module operations and
$T: X\to Y$ a triple module homomorphism. Then for every sequence
$(a_n)$ of mutually orthogonal  non-zero elements  in $E$, we
have:
\begin{enumerate}[$(a)$]
\item
$Q(a_n)^2 T$ is continuous for all but a finite number of $n$;
\item $a_n^{[3]}$ belongs to {\rm Ann}$_{E} (\sigma_{_Y}
(T))$ for all but a finite number of $n$;
\item  the set
$$\left\{ \frac{\|Q(a_n^{[3]}) T\|}{\|a_n\|^6} :  Q(a_n^{[3]}) T
\hbox{ is continuous }\right\}$$ is bounded.
\end{enumerate}
\end{lemma}

\begin{proof} Suppose that the statement (a) of the
lemma is false. Passing to a subsequence, we may assume that
$Q(a_n)^2 T$ is an unbounded operator for every natural $n$. In
this case we can find a sequence $(x_n)$ in $X$ satisfying
$\|x_n\| \leq 2^{-n} \|a_n\|^{-2},$ and $\|Q(a_n)^2 T(x_n) \| > n
\ K_n$, where $K_n$ is the norm of the bounded conjugate linear
operator $Q(a_n) : Y\to Y$, $Q(a_n) y = \J {a_n}y{a_n}$. Since
$Q(a_n)^2 T$ is discontinuous $K_n = \|Q(a_n)\|\neq 0,$ for every
$n$. (Note that  $\|Q(a)\|\le M\|a\|^2$ for some constant
$M$.)\smallskip

The series $ \sum_{k=1}^{\infty} Q(a_k) (x_k) $ defines an element
$z$ in the Banach triple module $X$. For $n\ne k$, the LOP and the
identity $$\J{x}{a_n}{\J{a_k}{a_n}{a_k}}
+\J{a_k}{\J{a_n}{x}{a_n}}{a_k}=$$
$$\J{\J{x}{a_n}{a_k}}{a_n}{a_k}+\J{a_k}{a_n}{\J{x}{a_n}{a_k}}
$$
shows that $\J{a_k}{\J{a_n}{x}{a_n}}{a_k}=0$. That is, $Q(a_k)Q(a_n)=0$ for $k\ne n$ and the same argument shows that for any $b\in E$,
\begin{equation}\label{755}
Q(a_k,b)Q(a_n)=0\hbox{ for }n\ne k.
\end{equation}
 Hence, for each  natural $n$, we
have
\begin{eqnarray*} K_n \|T(z)\| &\geq& \|Q(a_n)T(z)\| = \|T Q(a_n) (z) \|\\
%(\hbox{since $X$ has the LOP})\\
&=& \|T Q(a_n)^2 (x_n) \| = \|
Q(a_n) ^2T(x_n) \| > K_n \ n,
\end{eqnarray*} which is impossible. This proves (a).\smallskip

Since $Q(a_n)^2 T$ is continuous for all but a finite number of
$n$ and the module operations are continuous on $Y$, it follows
that $Q(a_n) Q(a_n)^2 T = Q(a_n)^3 T = Q(a_n^{[3]})T$ is continuous (and
hence,  $a_n^{[3]}\in$ {\rm Ann}$_{E} (\sigma_{_Y} (T))$) for all
but a finite number of $n$. This proves (b).\smallskip

In order to prove (c)  we may assume that $Q(a_n)^2 T$ is
continuous for every natural $n$. Arguing by reduction to the
absurd, we assume that  $\left\{ \frac{\|Q(a_n^{[3]})
T\|}{\|a_n\|^6} :  n\in \mathbb{N}\right\}$ is unbounded. There is
no loss of generality in assuming that $\|a_n\| = 1$, for every
$n$. By the Cantor diagonal process we may find a doubly indexed
subsequence $(a_{p,q})_{_{p,q\in \mathbb{N}}}$ of $(a_n)$ and a
doubly indexed sequence $(x_{p,q})$ in the unit sphere of $X$ such
that $\left\|Q(a_{p,q}^{[3]})\  T (x_{p,q})\right\| > 4^{2 q} \ q
\ p. $ Let $b_{p} := \sum_{q=1}^{+\infty}  2^{-q} \ a_{p,q}\in E$.
We observe that $a_{p,q} \perp a_{l,m}$ for every $(p,q)\neq
(l,m)$. It is therefore clear that $(b_p)$ is a sequence of
mutually orthogonal elements in $E$.  Having in mind that $X$
satisfies LOP, we deduce from (\ref{basic equation}) and (\ref{755})  that $Q(b_p)^2 Q(a_{p,q}) (x) = 4^{-2 q}
\ Q(a_{p,q}^{[3]}) (x),$ for every $x$ in $X$. Thus,
$$\left\| Q(b_{p})^2 T Q(a_{p,q}) (x_{p,q}) \right\| =\left\| T
Q(b_{p})^2 Q(a_{p,q}) (x_{p,q}) \right\| $$
$$= 4^{-2q} \left\|  T  Q(a_{p,q}^{[3]}) (x_{p,q}) \right\|=
4^{-2q} \left\|  Q(a_{p,q}^{[3]}) T (x_{p,q}) \right\| >q \ p,$$
for every $p,q$ in $\NN$, which shows that $Q(b_{p})^2 T$ is
unbounded for every $p\in \mathbb{N}.$ This contradicts the first
statement of the lemma and proves (c).
\end{proof}

Let $E$ be a complex (resp., real) Jordan triple and let $X$ be a
triple $E$-module. It is not hard to see that for every derivation $\delta: E\to X$
the mapping
$$\Theta_{\delta}  : E \to E\oplus X$$ $$a\mapsto a+\delta (a)$$ is a real linear Jordan
triple monomorphism between from the real Jordan triple $E$ to the triple split null extension $E\oplus X$.  (We observe
that, in this case, $E$ is regarded as a real Jordan triple
whenever it is a complex Jordan triple).\smallskip

When $X$ is a Jordan Banach triple $E$-module over a real or
complex JB$^*$-triple $E$, we define a norm, $\|.\|_0$, on the
triple split null extension of $E$ and $X$ by the assignment
$a+x\mapsto \|a+x\|_0:= \|a\|+\|x\|.$ The real Jordan triple
$E\oplus X$ becomes a real Jordan Banach triple. It is not hard to
see that, in this setting, a derivation $\delta$ is continuous if
and only if the triple monomorphism $\Theta_{\delta}$ is.
Moreover, the separating spaces $\sigma_{_X} (\delta)$ and
$\sigma_{_{E\oplus X}} (\Theta_{\delta})$ are linked by the the
following identity \begin{equation}\label{eq seprating spaces der
and triple monomorphism}\sigma_{_{E\oplus X}} (\Theta_{\delta}) =
\{0\}\times \sigma_{X} (\delta).
\end{equation}
Moreover,
\[
 {\rm Ann}_{E} (\sigma_{E\oplus X}
(\Theta_{\delta}))=\hbox{{\rm Ann}}_{E}
(\sigma_X(\delta)).
\]
\smallskip

The linear space $\Theta_{\delta} (E)$ is a subtriple of $E \oplus X$ and is made into
a  triple $E$-module
for the products  $$\J ab{\Theta_{\delta} (c)} = \Theta_{\delta} (\J abc ) =
\J {\Theta_{\delta} (a)}{\Theta_{\delta} (b)}{\Theta_{\delta} (c)}=
\J a{\Theta_{\delta} (b)}c,$$ $(a,b,c\in E)$. These products can be extended
to the $\|.\|_{0}$-closure, $\overline{\Theta_{\delta} (E)}$, of $\Theta_{\delta} (E)$.
Under this point of view, the mapping $\Theta_{\delta} :E \to \overline{\Theta_{\delta} (E)}$
is a  triple $E$-module homomorphism. The following result derives from the previous
Lemma \ref{l 0}, since $Q(a)\Theta_\delta=Q(a)\oplus Q(a)\delta$.

\begin{corollary}\label{c pre quotient is reflexive} Let $E$ be a
complex (resp., real) JB$^*$-triple, $X$ a Banach space which is a triple $E$-module with
continuous module operations and let $\delta: E\to X$ be a triple derivation. Then for every sequence
$(a_n)$ of mutually orthogonal  non-zero elements  in $E$,
$Q(a_n)^2 \delta$ is continuous for all but a finite number of $n$. It
follows that $a_n^{[3]}$ belongs to {\rm Ann}$_{E} (\sigma_{_X}
(\delta))$ for all but a finite number of $n$. Moreover, the set
$$\left\{ \frac{\|Q(a_n^{[3]}) \delta\|}{\|a_n\|^6} :  Q(a_n^{[3]}) \delta
\hbox{ is continuous }\right\}$$ is bounded.$\hfill\Box$
\end{corollary}

Let $E$ be a real or complex JB$^*$-triple. We shall say that
$E$ is \emph{algebraic} if all singly-generated subtriples
of $E$ are finite-dimensional. If in fact there exists $m\in \NN$
such that all single-generated subtriples of $X$ have dimension $\leq m$,
then $E$ is said to be of \emph{bounded degree}, and the minimum such
an $m$ will be called the \emph{degree} of $E$.

\begin{corollary}\label{c quotient is reflexive}%Corollary 7%
 Let $E$ be a
complex (resp., real) JB$^*$-triple, $X$ a Banach triple
$E$-module and let $\delta: E\to X$ be a triple derivation. Suppose
 that $\hbox{Ann}_{E}(\sigma_{_X} (\delta))$ is a norm closed
triple ideal of $E$. Then every element in
$E/\hbox{Ann}_{E}(\sigma_{_X} (\delta))$ has finite triple
spectrum, in other words, the JB$^*$-triple
$E/\hbox{Ann}_{E}(\sigma_{_X} (\delta))$ is isomorphic to a
Hilbert space or, equivalently, algebraic of bounded degree.
\end{corollary}

\begin{proof}
Let $\overline{a}$ be an element in the JB$^*$-triple
$F=E/\hbox{Ann}_{E}(\sigma_{_X} (\delta))$. Let $I_{a}$ denote the
intersection of $E_a$ with $\hbox{Ann}_{E}(\sigma_{_X} (\delta))$.
It is clear that $I_a$ is a norm closed triple ideal of $E_a$.
Moreover, the subtriple $F_{\overline{a}}$ is JB$^*$-triple
isomorphic to the quotient of $E_a$ with $I_a$.\smallskip

$E_a$ is JB$^*$-triple isomorphic (and hence isometric) to $C_0
(L)= C_0(L,\CC)$ (resp., $C_0 (L)= C_0(L,\RR)$) for some locally
compact Hausdorff space $L\subseteq (0,\|a\|]$ (called the triple spectrum of $a$)
such that $L\cup\{0\}$ is compact. We shall identify $E_a$ with $C_0(L)$.  It is known (compare
\cite[Proposition 3.10]{FriRus83}) that $E_a/I_a\cong C_0(\Lambda)$ where
 $$\Lambda =\{ t\in L : b
(t) = 0, \hbox{ for every } b\in I_a\}.$$

 We claim that the  set $\Lambda$ is finite. Otherwise,
there exists an infinite sequence $(t_n)$ in $\Lambda$.  We find a
sequence $(f_n) $ of mutually orthogonal elements in $C_0 (L)$
such that $f_n (t_n) \neq 0$ and hence $f_n\not\in I_a$ and $f_n^{[3]}\not \in I_a$.  Since orthogonality is a ``local''
concept, $(f_n)$ is a sequence of mutually orthogonal elements in
$E$ and $(f_n^{[3]})\not\in\hbox{Ann}_{E}(\sigma_{_X} (\delta)),$ we have a contradiction to  Corollary
\ref{c pre quotient is reflexive}.\smallskip

It follows that $E_a/I_a\cong F_{\overline{a}}$ is finite
dimensional.  The final statement follows from \cite[\S 4]{BuChu}
and \cite[\S 3, Theorems 3.1 and 3.8]{BeLoPeRo}.
\end{proof}

\subsection{Automatic continuity results}

Our main result  (Theorem \ref{t main complex to module}) will be proved in two steps, the first being the following proposition.

\begin{proposition}\label{p 1} Let $E$ be a complex (resp., real)
JB$^*$-triple, $X$ a Banach triple $E$-module, and let $\delta:
E\to X$ be a triple derivation. Assume that
$\hbox{Ann}_{E}(\sigma_{_X} (\delta))$ is a (norm-closed) linear
subspace of $E$ and that in the triple split null extension
$E\oplus X$,
\begin{equation}\label{eq 2 in prop p1}
\J {\hbox{Ann}_{_{E}} (\sigma_{_X}
(\delta))}{\hbox{Ann}_{_{E}} (\sigma_{_X} (\delta))}{\sigma_{_X}
(\delta)} = 0.\end{equation}
 Then $\delta|_{\hbox{Ann}_{E}(\sigma_{_X}
(\delta))} : \hbox{Ann}_{E}(\sigma_{_X} (\delta)) \to X$ is
continuous.
\end{proposition}

\begin{proof} By Lemma \ref{l separating spaces of a derivation}, $\sigma_{_X} (\delta)$
is a triple $E$-submodule of $X$. Since we are assuming  that $\hbox{Ann}_{E}(\sigma_{_X} (\delta))$ is a norm-closed
subspace of $E$, as we have seen, it is a norm-closed triple ideal of $E$.
\smallskip

Fix two arbitrary elements $a,b$ in $\hbox{Ann}_{E}(\sigma_{_X}
(\delta))$.  Since $a+b\in \hbox{Ann}_{E}(\sigma_{_X} (\delta))$, for every  $x$ in $\sigma_{_X} (\delta)$,
we have\[  2 \J a{x}b = \J
{a+b}{x}{a+b} - \J a{x}a - \J bxb =0,
\]

  Hence, in addition to our assumption (\ref{eq 2 in prop p1}),
 we also have
\[
 \J a{x}b =0, \hbox{ for every  } x \in
\sigma_{_X} (\delta),\ a,b\in \hbox{Ann}_{E}(\sigma_{_X}
(\delta)),
\]
that is,
\begin{equation}
\label{eq 1 in prop p1}
\J {\hbox{Ann}_{_{E}} (\sigma_{_X} (\delta))}   {\sigma_{_X} (\delta)}{\hbox{Ann}_{_{E}} (\sigma_{_X} (\delta)} = 0.
\end{equation}

Considering $L(a,b)$ and $Q(a,b)$ as linear mappings from $X$ to
$X$ defined by $L(a,b) (x) = \J abx$ and $Q(a,b) (x) = \J axb$
($x\in X$), we deduce from  $(\ref{eq 2 in prop p1})$, $(\ref{eq 1
in prop p1})$ that $\sigma_X(\delta)\subset \ker L(a,b)\cap \ker
Q(a,b)$ and therefore that $L(a,b) \delta, Q(a,b) \delta : E \to
X$ are continuous operators for every  $a,b\in
\hbox{Ann}_{E}(\sigma_{_X} (\delta))$.
\smallskip

When $L(a,b)$ and $Q(a,b)$ as considered as (real) linear operators from $E$ to $E$,
the compositions $\delta L(a,b) $ and $\delta Q(a,b)$ satisfy the identities
\begin{eqnarray*}\delta L(a,b) (c)& =& \J {\delta(a)}{b}c + \J a{\delta(b)}c + \J ab{\delta(c)}\\
&=& L(\delta a,b)(c)+L(a,\delta b)(c)+L(a,b)\delta (c)
\end{eqnarray*}
and
\begin{eqnarray*}
\delta Q(a,b) (c)& =& \J {\delta(a)}{c}b + \J a{\delta(c)}b + \J ac{\delta(b)}\\
&=& Q(\delta a,b)(c)+Q(a,b)\delta (c)+Q(a,\delta b)(c).
\end{eqnarray*}
for an arbitrary $c\in E$. Since $X$ is a Banach triple
$E$-module, the continuity of $L(a,b) \delta$ and $ Q(a,b) \delta$
as operators from  $E$ to $X$ implies that the mappings $c\mapsto
\delta (\J abc)$ and $c\mapsto \delta (\J acb)$ are continuous
linear operators from $E$ to $X$.\smallskip

Let $W: \hbox{Ann}_{E}(\sigma_{_X} (\delta))\times
\hbox{Ann}_{E}(\sigma_{_X} (\delta))\times
\hbox{Ann}_{E}(\sigma_{_X} (\delta)) \to X$  be the real trilinear
mapping defined by $ W(a,b,c) := \delta (\J abc)$. We have already
seen that $W$ is separately continuous whenever we fix two of the
variables in $(a,b,c)\in \hbox{Ann}_{E}(\sigma_{_X}
(\delta))\times \hbox{Ann}_{E}(\sigma_{_X} (\delta))\times
\hbox{Ann}_{E}(\sigma_{_X} (\delta))$. By repeated application of the uniform boundedness
principle, $W$ is (jointly) continuous. Therefore, there exists a positive
constant $M$ such that $\|\delta \J abc \| \leq M\ \|a\| \ \|b\| \
\|c\|$, for every $a,b,c\in \hbox{Ann}_{E}(\sigma_{_X} (\delta))$.

Finally, since $\hbox{Ann}_{E}(\sigma_{_X} (\delta))$ is a
JB$^*$-subtriple of $E$, for each $a$ in \linebreak
$\hbox{Ann}_{E} (\sigma_{_X} (\delta))$, there exists $b$ in
$\hbox{Ann}_{E}(\sigma_{_X} (\delta))$ satisfying that $b^{[3]}
=a$. In this case $$\|\delta (a) \| = \|\delta \J bbb \| \leq M \
\|b\|^3 = M \|\J bbb\| = M \ \|a\|,$$ which shows that the
restriction of $\delta$ to ${\hbox{Ann}_{E}(\sigma_{_X} (\delta))}$ is
continuous.
\end{proof}

We can state now the main results of the paper.

\begin{theorem}\label{t main complex to module} Let $E$ be a complex
JB$^*$-triple, $X$ a Banach triple $E$-module, and let $\delta:
E\to X$ be a triple derivation. Then $\delta$ is continuous if and
only if $\hbox{Ann}_{E}(\sigma_{_X} (\delta))$ is a (norm-closed)
linear subspace of $E$ and $$\J {\hbox{Ann}_{_{E}} (\sigma_{_X}
(\delta))}{\hbox{Ann}_{_{E}} (\sigma_{_X} (\delta))}{\sigma_{_X}
(\delta)} = 0,$$ in the triple split null extension $E\oplus X$.
\end{theorem}

\begin{proof}
If $\delta$ is continuous $\hbox{Ann}_{E}(\sigma_{_X} (\delta)) =
 \hbox{Ann}_{E}(\{0\}) = E$ is a linear subspace of $E$ and $\J EE{0} = 0$.

Conversely, let us suppose that $E$ is a complex JB$^*$-triple and
that\linebreak $\hbox{Ann}_{E}(\sigma_{_X} (\delta))$ is a
norm-closed subspace of $E$ and hence a norm-closed triple ideal
of $E$.\smallskip

 In order to
simplify notation, we denote $J = \hbox{Ann}_{E}(\sigma_{_X}
(\delta))$, while the projection of $E$ onto $E/J$ will be denoted
by $a\mapsto \pi (a)=\overline{a}$.\smallskip

By Corollary \ref{c quotient is reflexive}, $E/J$ is algebraic of
bounded degree $m$. Thus, for each element $\overline{a}$ in $E/J$
there exist mutually orthogonal minimal tripotents
$\overline{e}_1,\ldots, \overline{e}_k$ in $E/J$ and
$0<\lambda_1\leq \ldots \leq \lambda_k$ with $k\leq m$ such that
$\overline{a} = \sum_{j=1}^{k} \lambda_j \overline{e}_j$. We shall
show in the next two paragraphs that
 $e_1,\ldots, e_k \in J$, and
hence, $a\in J$. This will show that $E=J$ and application of Proposition~\ref{p 1} will complete  the proof.

Suppose that $\overline{e}$ is a minimal tripotent in $E/J$, where
$e\in E$ is a representative in the class $\overline{e}$. In this
case $\left(E/J\right)_2 (\overline{e}) = \CC \overline{e}$. Take
an arbitrary sequence $(a_n)$ converging to $0$ in $E$. For each
natural $n$, there exists a scalar $\mu_n\in \CC$ such that $$\pi
(Q(e) (a_n) ) = Q(\overline{e}) (\pi (a_n)) =Q(\overline{e})
(\overline{a}_n) = \mu_n \overline{e} =\pi (\mu_n e).$$ The
continuity of $\pi$ and the Peirce 2 projection $P_{2}
(\overline{e})$ assure that $\mu_n \to 0$. It follows that the
sequence $Q(e) (a_n) - \mu_n e$ lies in $J$ and tends to zero in
norm.\smallskip

By Proposition \ref{p 1}, $\delta|_{J}$ is continuous. Therefore,
$$\delta (Q(e) (a_n) )= \delta (Q(e) (a_n) - \mu_n e) +\mu_n
\delta (e) \to 0.$$ Since $(a_n)$ is an arbitrary norm null
sequence in $E$, the linear mapping $\delta Q(e) : E \to X$ is
continuous, and hence $e\in \hbox{Ann}_{E}(\sigma_{_X} (\delta)) =
J$, or equivalently, $\overline{e} =0$.\smallskip
\end{proof}

Let $E$ be a real JB$^*$-triple.  By \cite[Proposition
2.2]{IsKaRo95}, there exists a unique complex JB$^*$-triple
structure on the complexification $\widehat {E} = E\oplus i\ E$,
and a unique conjugation (i.e., conjugate-linear isometry of
period 2) $\tau$ on $\widehat {E}$ such that $E=\widehat {E}
^{\tau} := \{x\in \widehat {E} : \tau (x)=x \}$, that is, $E$ is a
real form of a complex JB$^*$-triple. Let us consider $$
\tau^\sharp : \E^{*} \rightarrow \E^{*}$$ defined by
$$\tau^\sharp (f) (z) = \overline{f (\tau (z))}.$$ The mapping $\tau^\sharp$
is a conjugation on $\E^{*}$. Furthermore the map
$$(\E^{*})^{\tau^\sharp}
\longrightarrow (\E^{\tau})^{*}\ (=E^*)$$ $$f \mapsto f|_{E}$$ is an
isometric bijection, where $(\E^{*})^{\tau^\sharp} := \{ f \in
\E^{*} : \tau^\sharp (f)=f \}$ (compare \cite[Page
316]{IsKaRo95}).\smallskip

\begin{remark}
\label{r complexification of a derivation}{\rm Let $\delta : E \to E^*$ be a triple derivation
from a real JB$^*$-triple to its dual.
It is not hard  (but tedious) to see that, under the identifications given in the above paragraph,
the mapping $\widehat{\delta} : \widehat {E} \to \widehat{E}^*$,
$\widehat{\delta} (x+i y) := \delta (x) - i \delta (y)$ is conjugate-linear and a
 triple derivation from $\widehat {E}$ to
$\widehat{E}^*$, when the latter is seen as a triple $E$-module.\smallskip

Actually, although the calculations are tedious, the triple products of every
real triple $E$-module, $X$, can be appropriately extended to its algebraic
complexification $\widehat{X} = X\oplus i X$ to make the latter a complex triple
$\widehat{E}$-module. Further, every (real linear) triple derivation $\delta : E \to X$ can be
extended to a (conjugate linear) triple derivation $\widehat{\delta} : \widehat {E} \to \widehat{X}$.}
\end{remark}

\begin{corollary}
\label{c automatic continuity} Let $E$ be a real or complex
JB$^*$-triple. \begin{enumerate}[$(a)$] \item Every derivation
$\delta : E \to E$ is continuous. \item Every derivation $\delta :
E \to E^*$ is continuous.
\end{enumerate}
\end{corollary}

\begin{proof}
The proof in the complex case follows now from Proposition \ref{p annihilator submodule dual}
and Theorem \ref{t main complex to module}.  (In Theorem~\ref{t main complex to module}, we consider $E$ as a real triple and as a real $E$-module, and $\delta$ as a real-linear map.) The statements in the real setting are,
by Remark \ref{r complexification of a derivation}, direct consequences of the corresponding results
in the complex case.
\end{proof}

The first statement of the above corollary was already established
in \cite[Corollary 2.2]{BarFri} and \cite[Remark 1]{HoMarPeRu}.
The proof given here is completely independent.
The second statement is new and will be important for a forthcoming study by the authors of weak amenability for JB$^*$-triples.
\smallskip

Recall that every derivation of a complex C$^*$-algebra $A$ into
a  Banach $A$-bimodule is automatically continuous \cite{Ringrose72}.
The class of Banach triple modules over real or complex JB$^*$-triples
is strictly wider  than the class of Banach  bimodules over C$^*$-algebras.
Our next remark shows that, in the more general setting of triple derivations
from real or complex JB$^*$-triples to Banach triple modules the continuity  is not,
in general, automatic.

\begin{remark}\label{r last}{\rm Let $H$ be a real Hilbert space with inner product denoted by $(.,.)$.
Suppose that  dim$(H)\geq 2$. Let $J$ denote the Banach space $\CC 1 \oplus^{\ell_1} H$.
It is known that $J$ is a JB-algebra with respect to the product $$(\lambda_1 1+ a_1) \circ (\lambda_2 1+ a_2):= \lambda_1  a_2 + \lambda_2  a_1 + (\lambda_1 \lambda_2 + (a_1,a_2) ) 1.$$ The JB-algebra $(J,\circ)$ is called a \emph{spin factor} (see \cite{Hanche}). It follows that $J$ is a real JB$^*$-triple via $\J abc := (a\circ b) \circ c + (c\circ b) \circ a - (a\circ c) \circ b, (a,b,c\in J).$\smallskip

It was already noticed by Hejazian and Niknan (see \cite[Definition 3.2]{HejNik96}) that every Banach space $X$ can be considered as a (degenerate) Jordan $J$-module with respect to the products $$ (\lambda_1 1+ a_1) \circ x= x \circ (\lambda_1 1+ a_1)= \lambda_1 x, \  (x\in X, \lambda_1\in \mathbb{R}, a_1\in H). $$  Since every linear mapping $D: J \to X $ with $D(1) =0$ is a Jordan derivation (i.e. $D( a \circ b) = D(a)\circ b + a\circ D(b)$, $\forall a,b\in J$), for every infinite dimensional spin factor $J$, there exists a discontinuous derivation from $J$ to a degenerate Jordan $J$-module. \smallskip

Each degenerate Banach Jordan $J$-module $X$ is a Banach triple $J$-module with respect to  $\J abx :=  (a\circ b) \circ x + (x\circ b) \circ a -(a\circ x) \circ b$ and  $\J axb = (a\circ x) \circ b + (b\circ
x) \circ a - (a\circ b) \circ x$ ($a,b\in J$, $x\in X$), and each linear mapping $\delta : J \to X$ with $\delta (1) =0$  is a triple derivation. Thus, for each infinite dimensional spin factor $J$ there exists a discontinuous triple derivation from $J$ to a Banach triple $J$-module.
}\end{remark}

\section{Derivations on a C$^*$-algebra}

A celebrated result due to J.R. Ringrose establishes that every
(associative) derivation from a C$^*$-algebra $A$ to a Banach
$A$-bimodule is continuous (cf. \cite{Ringrose72}). We have
already commented that S. Hejazian and A. Niknam gave in \cite[\S
3]{HejNik96} an example of a discontinuous Jordan derivation from
a JB$^*$-algebra to a Jordan Banach module. Based on this example,
we have already shown the existence of a discontinuous triple
derivation from a JB$^*$-triple to a Banach triple module (see
Remark \ref{r last}). The aim of this section is to show that
these two  counterexamples cannot be found when the domain is a
C$^*$-algebra, thereby providing positive answers to Problems 1
and 2 (Corollaries 20 and 21). We shall also  see that Ringrose's
Theorem derives as a consequence of our results (Corollary 22).
\smallskip

We shall actually prove a stronger result: every triple derivation from a  C$^*$-algebra to a Banach triple module is automatically continuous (Theorem 19),  which will imply these three corollaries.\smallskip

% It will follow, as a consequence,
%that every Jordan (resp., triple derivation)
%from a C$^*$-algebra $A$ to a Banach $A$-bimodule is automatically
%continuous, which answers a question posed in \cite[Question
%14.i.]{Vill} and \cite{AlBreVill}.\smallskip

We shall need a technical reformulation of
Theorem \ref{t main complex to module} above. Theorem \ref{t main complex to module}
has been established only for complex JB$^*$-triples.
The proof given in Section 3 is not valid for real JB$^*$-triples.
The obstacles appearing in the real setting concern the structure of the Peirce-2
subspace associated with a minimal tripotent.
We have already commented that, in case of $E$ being a
complex JB$^*$-triple, the identity $E^{-1} (e) = i E^{1} (e)$
holds for every tripotent $e$ in $E$, whereas in the real
situation the dimensions of $E^{1} (e)$ and $E^{-1} (e)$ are not,
in general, correlated. For example, every infinite dimensional
rank-one real Cartan factor $C$ contains a minimal tripotent $e$ satisfying that
$C^{1} (e) = \RR e$ and $\dim (C^{-1} (e)) = +\infty$ (compare \cite[Remark 2.6]{FerMarPe}).
\smallskip

Following \cite[11.9]{Loos77}, we shall say that a real
JB$^*$-triple $E$ is \emph{reduced} whenever $E_{2} (e) = \RR e$
(equivalently, $E^{-1} (e) = 0$) for every minimal tripotent $e\in
E$. Reduced real Cartan factors were studied and classified in
\cite[11.9]{Loos77} and in \cite[Table 1]{Ka97}. Reduced real
JB$^*$-triples played an important role in the study of the
surjective isometries between real JB$^*$-triples developed in
\cite{FerMarPe}.\smallskip

Having the above comments in mind, it is not hard to check that,
in the particular subclass of reduced real JB$^*$-triples the
proof of Theorem \ref{t main complex to module} remains valid line
by line. We therefore have:

\begin{proposition}\label{thm PeRu real} Let $E$ be a reduced real
JB$^*$-triple, $X$ a Banach triple $E$-module, and let $\delta:
E\to X$ be a triple derivation. Then $\delta$ is continuous if and
only if $\hbox{Ann}_{E}(\sigma_{_X} (\delta))$ is a (norm-closed)
linear subspace of $E$ and $$\J {\hbox{Ann}_{_{E}} (\sigma_{_X}
(\delta))}{\hbox{Ann}_{_{E}} (\sigma_{_X} (\delta))}{\sigma_{_X}
(\delta)} = 0,$$ in the triple split null extension $E\oplus
X$.$\hfill\Box$
\end{proposition}

Every closed ideal of a reduced real JB$^*$-triple is a reduced
real JB$^*$-triple. It is also true that the self-adjoint part,
$A_{sa}$, of a C$^*$-algebra, $A$, is a reduced real JB$^*$-triple
with respect to the product
\begin{equation}\label{eq triple product self-adjoint part}
\J abc := \frac{a bc +cba}{2} \ \ (a,b,c\in A_{sa}),\end{equation}
or equivalently,
\begin{equation}\label{eq triple product self-adjoint part jordan}
\J abc := (a\circ b) \circ c +  (c\circ b) \circ a -  (a\circ c)
\circ b,\ \ (a,b,c\in A_{sa}).\end{equation}
Indeed, writing $e=p-q$ for a minimal partial isometry $e\in A_{sa}$ with $p$ and $q$ orthogonal projections, it is easy to check that $e=p$ or $e=-q$ and it follows that if $exe=-x$, then $x=0$.
 In particular, for
each closed triple ideal $J$ of $A_{sa}$, the quotient
$A_{sa}/J$ is a reduced real JB$^*$-triple.

\smallskip

Our next result is a consequence of the previous proposition. Note that the fact that $A_{sa}$ is a reduced JB$^*$-triple is only needed in the case that $A$ is an abelian C$^*$-algebra. \smallskip

\begin{proposition}
\label{prop ternary derivations from abelian C*-algebras to triple modules}
Let $A$ be an abelian C$^*$-algebra whose self adjoint part
is denoted by $A_{sa}$. Then, every triple derivation from $A_{sa}$
to a real Jordan-Banach triple $A_{sa}$-module is continuous.
In particular, every triple derivation from $A$
into a real Jordan-Banach triple $A$-module is continuous.
\end{proposition}

\begin{proof}
Let $\delta : A_{sa}\to X$ be a triple derivation from
$A_{sa}$ into a real Jordan triple $A_{sa}$-module.
The statement of the proposition will follow from
Proposition \ref{thm PeRu real} as soon as we prove that
$\hbox{Ann}(\sigma_{_X} (\delta))=\hbox{Ann}_{A_{sa}}
(\sigma_{_X} (\delta))$ is a (norm-closed)
linear subspace of $A_{sa}$ and $$\J {\hbox{Ann} (\sigma_{_X}
(\delta))}{\hbox{Ann} (\sigma_{_X} (\delta))}{\sigma_{_X}
(\delta)} = 0.$$

Let us take $a\in \hbox{Ann} (\sigma_{_X} (\delta))$.
Having in mind that
$a\in \hbox{Ann} (\sigma_{_X} (\delta))$ if, and only if,
$Q(a) \delta$ (or equivalently, $\delta Q(a)$) is a continuous
operator from $A_{sa}$ to $X$ (see the comments after
Lemma \ref{l separating spaces of a derivation}),
we observe that $\delta Q(a)$ is a continuous mapping
from $A_{sa}$ to $X$. Obviously, for each $b$ in $A_{sa}$, the operator $L_{b}:A_{sa} \to A_{sa}$,
$c\mapsto c b = bc$ is continuous. Since $A$ is abelian we have $L(a^2,b) = Q(a) L_{b}= L_{b} Q(a)$.
Therefore $\delta L(a^2,b) = \delta Q(a) L_{b}$ is a continuous operator from $A_{sa}$ to $X$. The identity
$$ \delta L(a^2,b) = L(\delta(a^2),b) + L(a^2,\delta(b)) + L(a^2,b) \delta$$ shows that $L(a^2,b) \delta$
is continuous. It is easy to check, from the definition of $\sigma_{_X} (\delta),$ that $\J {a^2}{b}{x}=0,$
for every $x\in \sigma_{_X} (\delta)$. It follows that \begin{equation}
\label{eq 1 prop abelian} \J {a^2}{b}{x}=0, \hbox{ for every $a\in \hbox{Ann} (\sigma_{_X} (\delta))$, $b\in A_{sa}$ and $x\in \sigma_{_X} (\delta)$.}
\end{equation}

It is known that $a$ can be written in the form $a= a_1 - a_2$, where $a_1$
and $a_2$ are two orthogonal positive elements in $A_{sa}$.
It is also known that $Q(a) (A_{sa})\in \hbox{Ann} (\sigma_{_X} (\delta))$.
Therefore, $a_1^3 = Q(a) (a_1) \in \hbox{Ann} (\sigma_{_X} (\delta))$ and hence $a_1^6 A_{sa}
= Q(a_1^3) (A_{sa}) \subseteq \hbox{Ann} (\sigma_{_X} (\delta))$. This implies that the ideal of $A_{sa}$
generated by $a_1^6$ lies in $\hbox{Ann} (\sigma_{_X} (\delta))$, which guarantees that
$a_1\in \hbox{Ann} (\sigma_{_X} (\delta))$. We can similarly show that $a_2$ belongs to
$\hbox{Ann} (\sigma_{_X} (\delta))$. A similar argument shows that $a_1^{\frac{1}{2}},
a_2^{\frac{1}{2}}\in \hbox{Ann} (\sigma_{_X} (\delta))$. Now, we deduce from $(\ref{eq 1 prop abelian})$
that \begin{equation}\label{eq 2 prop abelian}  \J {a}{b}{x} =\J {a_1}{b}{x} - \J {a_2}{b}{x} =0,
\end{equation} for every $a\in \hbox{Ann} (\sigma_{_X} (\delta))$, $b\in A_{sa}$ and $x\in \sigma_{_X} (\delta)$,
or equivalently, $\delta L(a,b)$ and $L(a,b) \delta$ are continuous operators
for every $a\in \hbox{Ann} (\sigma_{_X} (\delta))$ and $b\in A_{sa}$.\smallskip

Since $A$ is abelian, $L(a,b) = Q(a,b)$ in $A_{sa}$, it follows from $(\ref{eq 2 prop abelian})$,
that $\delta Q(a,b)$ and $Q(a,b) \delta$ are continuous operators from $A_{sa}$ to $X$
for every $a\in \hbox{Ann} (\sigma_{_X} (\delta))$ and $b\in A_{sa}$.
This implies that \begin{equation}
\label{eq 3 prop abelian}  \J {a}{x}{b} =0,
\hbox{ for every $a\in \hbox{Ann} (\sigma_{_X} (\delta))$, $b\in A_{sa}$ and $x\in \sigma_{_X} (\delta)$.}
\end{equation}

Finally, given $a,c$ in $\hbox{Ann} (\sigma_{_X} (\delta))$, we deduce from $(\ref{eq 3 prop abelian})$
that $$Q(a+c) (\sigma_{_X} (\delta)) = Q(a) (\sigma_{_X} (\delta)) + Q(c) (\sigma_{_X} (\delta))
+ 2 Q(a,c) (\sigma_{_X} (\delta)) =0,$$ which shows that $a+c\in \hbox{Ann} (\sigma_{_X} (\delta))$,
and hence the latter is a linear subspace of $A_{sa}$.\end{proof}

Given any element $x$ in a C$^*$-algebra $A$, we shall denote by $C(x)$
the C$^*$-subalgebra of $A$ generated by $x$.\smallskip

The following theorem, due to J. Cuntz (see \cite{Cuntz}) will be required later.

\begin{lemma}
\label{Thm Cuntz}\cite[Theorem 1.3]{Cuntz} Let $A$ be a C$^*$-algebra and $f$ a linear functional on $A$. If $f$ is
continuous on $C(h)$ for all $h=h^*$ in $A$, then $A$ is continuous on $A$. By the uniform boundedness theorem,
a linear mapping $T$ from $A$ to a normed space $X$ is continuous if and only if it restriction to
$C(h)$ is continuous for all $h=h^*$ in $A$.  $\hfill\Box$
\end{lemma}

Let $\delta: A \to X$ be a triple derivation from a C$^*$-algebra
to a Banach triple $A$-module. For each self-adjoint element $h$
in $A$, the Banach space $X$ can be regarded as a Jordan Banach
$C(h)$-module by restricting the module operation from $A$ to
$C(h)$. Since $\delta|_{C(h)}: C(h) \to X$ is a triple derivation
from an abelian C$^*$-algebra into a Banach triple $C(h)$-module,
Proposition \ref{prop ternary derivations from abelian C*-algebras
to triple modules} assures that $\delta|_{C(h)}$ is continuous.
Combining this argument with the above Cuntz's theorem we have:

\begin{theorem}
\label{cor ternary derivations from C*-algebras} Let $A$ be a C$^*$-algebra.
Then every triple derivation from $A$ (respectively, from $A_{sa}$) into a complex
(respectively, real) Jordan Banach triple $A$-module is continuous.$\hfill\Box$
\end{theorem}

Since every Jordan derivation is a triple derivation, and every
Jordan module is a Jordan triple module, we have:\smallskip

\begin{corollary}[Solution to Problem 2]
\label{cor Jordan derivations from C*-algebras}
Let $A$ be a C$^*$-algebra. Then every Jordan derivation
from $A$ into a Jordan-Banach $A$-module   $X$
is continuous.$\hfill\Box$
\end{corollary}

It is due to B.E. Johnson that every continuous Jordan derivation
from a C$^*$-algebra $A$ to a Banach $A$-bimodule is a derivation
(cf. \cite[Theorem 6.2]{John96}). As we have just seen, the
hypothesis of continuity can be omitted in the just quoted
theorem.  Thus: \smallskip

\begin{corollary}[Solution to Problem 1]
\label{cor associative derivations from C*-algebras}
Let $A$ be a C$^*$-algebra. Then every Jordan derivation
from $A$ into a Banach $A$-bimodule  $X$
is continuous. In particular, every Jordan derivation
from $A$ to $X$ is a derivation, by Johnson's theorem.$\hfill\Box$
\end{corollary}

Let $D: A\to X$ be an associative (resp., Jordan) derivation from
a C$^*$-algebra to a Banach $A$-bimodule. The space $X$, regarded
as a real Banach space, is a real Banach triple $A_{sa}$-module
with respect to the product defined in $(\ref{eq triple product
self-adjoint part jordan})$, where, in this case, one element in
$(a,b,c)$ is taken in $X$ and the other two in $A_{sa}$. The
restriction of $D$ to $A_{sa}$, $\delta = D|_{A_{sa}} : A_{sa} \to
X$ is a (real linear) triple derivation. Hence, Theorem \ref{cor
ternary derivations from C*-algebras} implies that $\delta$ (and
hence $D$) is continuous.  Thus: \smallskip

\begin{corollary}[Ringrose]
\label{Ringrose}
Let $A$ be a C$^*$-algebra. Then every  derivation
from $A$ into a Banach $A$-bimodule  $X$
is continuous.$\hfill\Box$
\end{corollary}

In \cite{HaaLaust}, U. Haagerup and N.J. Laustsen presented a new
proof of Johnson's Theorem. Applying a result of automatic
continuity in \cite[Corollary 2.3]{HejNik96}, the just quoted
authors proved that every Jordan derivation from a C$^*$-algebra
$A$ to $A^*$ is bounded and hence an inner derivation (cf.
\cite[Corollary 2.5]{HaaLaust}).\smallskip

In \cite{Bre05},  M. Bre\v{s}ar studied a more general class of
Jordan derivations from a C$^*$-algebra $A$ to an $A$-bimodule
$X$. An additive mapping $d: A \to X$ satisfying $d(a\circ b) =
d(a)\circ b + a \circ d(b),$ for every $a,b\in A$, is called an
\emph{additive Jordan derivation}. An additive Jordan derivation
is said to be \emph{proper} when it is not an associative
derivation. Every (linear) Jordan derivation $D: A \to X$ is an
additive Jordan derivation. However, the reciprocal implication
is, in general, false. Actually, from \cite[Theorem 5.1]{Bre05},
for each unital C$^*$-algebra $A$, then there exists a proper
additive Jordan derivation from $A$ into some unital $A$-bimodule
if, and only if, $A$ contains an ideal of codimension one.

\medskip\medskip

\end{document}